\documentclass[a4paper,twoside,openbib,10pt]{article}

\usepackage{vmargin}

\setpapersize{A4}
\setmargins{4.3cm}       
{1.5cm}                        
{12.5cm}                      
{23.42cm}                    
{10pt}                           
{1cm}                           
{0pt}                             
{2cm}

\usepackage{amsfonts}
\usepackage{amsmath}
\usepackage{amsthm}
\usepackage{amssymb}
\usepackage[all]{xy}
\usepackage[latin1]{inputenc}
\usepackage{graphicx}
\usepackage{color}
\usepackage{yfonts}
\usepackage{emptypage}
\usepackage{hyperref}
\input xy
\xyoption{all}

\usepackage{graphicx}

\usepackage{mathrsfs}
\usepackage{float}
\usepackage{framed}
\usepackage{titlesec}

\newtheorem{theorem}{Theorem}[section]
\newtheorem{corollary}[theorem]{Corollary}
\newtheorem{lemma}[theorem]{Lemma}
\newtheorem{proposition}[theorem]{Proposition}
\theoremstyle{definition}
\newtheorem{definition}[theorem]{Definition}
\newtheorem{remark}[theorem]{Remark}

\newtheorem*{notation}{Notation}

\newcommand{\B}{\mathcal{B}}

\newcommand{\G}{\mathcal{G}}
\newcommand{\C}{\mathcal{C}}

\newcommand{\Oy}{\mathscr{O}_{Y}}
\newcommand{\Ox}{\mathscr{Oo}_{C}}

\newcommand{\slas}{/\!\!/ }

\newcommand{\N}{\mathcal N}

\newcommand{\CC}{\mathscr C}
\newcommand{\F}{\mathscr{F}}
\newcommand{\E}{\mathscr{E}}
\newcommand{\Oo}{\mathscr{O}}

\newcommand{\Ss}{\mathbf{Spec}}
\newcommand{\BB}{\mathscr B}
\newcommand{\LL}{\mathscr L}

\titleformat{\section}[display]{\scshape \bfseries}{\S \thesection\filcenter}{1ex}{\fillast}

\titleformat{\subsection}[hang]{\itshape\bfseries}{\thesubsection. --- }{0pt}{\upshape}

\titleformat{\subsubsection}[runin]{\itshape\bfseries}{\thesubsubsection. }{0pt}{\upshape}

\usepackage{fancyhdr}
\usepackage{lipsum}
\usepackage{minitoc}

\newcommand{\eq}[1][r]
   {\ar@<-3pt>@{-}[#1]
    \ar@<-1pt>@{}[#1]|<{}="gauche"
    \ar@<+0pt>@{}[#1]|-{}="milieu"
    \ar@<+1pt>@{}[#1]|>{}="droite"
    \ar@/^2pt/@{-}"gauche";"milieu"
    \ar@/_2pt/@{-}"milieu";"droite"}

\date{}
\title{\textbf{On the moduli space of singular principal $G$-bundles over stable curves}
\author{{\footnotesize ÁNGEL LUIS MU\~NOZ CASTA\~NEDA}\footnote{RIASC Universidad de Le\'on, Spain, mail: {\texttt amunc@unileon.es}}}
}

\begin{document}

\maketitle

\begin{center}
\begin{minipage}[b][1.5\height][t]{0.8\textwidth}
{\small {\normalsize A}{\scriptsize BSTRACT.} \ In this paper we proof the existence of a linearization for singular principal $G$-bundles not depending on the base curve. This allow us to construct the relative compact moduli space of $\delta$-(semi)stable singular principal G-bundles over families of reduced projective and connected nodal curves, and to reduce the construction of the universal moduli space over $\overline{\mathcal{M}}_{g}$ to the construction of the universal moduli space of swamps..}
\end{minipage}
\end{center}




\tableofcontents

\section{\footnotesize INTRODUCTION}
\markboth{\emph{}}{} 

This work\footnote{This paper has been submitted to a peer-review journal} is focused on the construction of a compactification of the moduli space of principal $G$-bundles on a reduced projective and connected nodal curve. A. Schmitt realised in \cite{Alexander2}  that a principal $G$-bundle can always be seen as a pair formed by a vector bundle and certain morphism of sheaves of algebras, once a fully faithful representation $\rho\colon G\hookrightarrow \textrm{SL}(n)$ is fixed. He constructed the compact moduli space of $\delta$-semistable singular principal $G$-bundles over smooth projective varieties. Although this moduli space depends a priori on the fixed representation, he proved that coincides with the classical moduli space constructed by A. Ramanathan, as long as the rational parameter $\delta$ is large enough. 
Afterwards, he generalized this construction to the case of an irreducible nodal curve with one node \cite{Alexander1}, while T. G\'omez and I. Sols simplified and generalized A. Ramanathan's construction for higher dimensional smooth projective varieties \cite{sols2}. In 2013, A. Langer \cite{L}, gave a construction of a compact moduli space of $\delta$-semistable singular principal $G$-bundles in the case of a family of  irreducible projective varieties. 

The construction presented in this work will be based on the techniques developed in \cite{bosle, sols, L, Alexander2, Alexander4, Alexander1}.

The structure of the paper is as follows. In Section \ref{section2}  we recall some basic definitions and known results. Section \ref{section3}  is devoted to show that the local structure of a torsion free sheaf on a nodal curve can be enlarged to certain open affine neighborhoods (Theorem \ref{localstructure}). In Section \ref{section4} we show how many pieces of the symmetric algebra $S^{\bullet}(V\otimes\F)^{G}$ we need to generate it (Theorem \ref{gen}).  It is shown that the number of required pieces depends not on the curve or the sheaf, but on the genus of the curve and the rank of the sheaf. The main novelty of this paper is Theorem \ref{C}; results of Section \ref{section3} and Section \ref{section4} allow us to prove that the type $(a,b)$ of the associated swamp of a singular principal $G$-bundle on a nodal curve does not depend on the base curve. 
This, in turn, allows to prove, in Section \ref{section6}, the existence of the moduli space of $\delta$-semistable singular principal $G$-bundles over a family of reduced projective and connected nodal curves (Theorem \ref{D}).
In Section \ref{section7} we point out that Theorem \ref{C} and Theorem \ref{D} reduce the construction of the universal moduli space of $\delta$-semistable singular principal $G$-bundles to the construction of the universal moduli space of $\delta$-semistable swamps. This might be a good candidate for the compactification of the universal moduli space of principal $G$-bundles over $\overline{\mathcal{M}}_{g}$ \cite{mum-ir,gies}.

\section{\footnotesize PRELIMINARIES}\label{section2}
\markboth{\emph{}}{} 

\subsection{Basic definitions and notation}

Let $\Oo$ be a local noetherian ring with residue field $k$ and $M$ a finitely generated $\Oo$-module. The depth of $M$ is defined as 
$
\textrm{depth}(M):=\textrm{min}\{i: \textrm{Ext}^{i}_{\Oo}(k,M)\neq 0\}.
$
Its dimension $\textrm{dim}(M)$ is just the dimension of its support, $\textrm{supp}(M)\subset \textrm{Spec} \ \Oo$. We say that $M\neq 0$ is a Cohen-Macaulay module if $\textrm{depth}(M)=\textrm{dim}(M)$. The ring $\Oo$ is Cohen-Macaulay if it is a Cohen-Macaulay module. If $R$ is an arbitrary noetherian ring, then an $R$-module $M$ is Cohen-Macaulay if $M_{\mathfrak{m}}$ is a Cohen-Macaulay $R_{\mathfrak{m}}$-module for all $\mathfrak{m}\in \textrm{supp}(M)$. If $X$ is a scheme and $\F$ a coherent $\Oo_{X}$-module we can extend the definition of Cohen-Macaulay to $X$ and $\F$ in the obvious way.

Let $C$ be a reduced connected projective curve over $\mathbb{C}$ (and therefore Cohen-Macaulay) and let $\F$ be a coherent $\Oo_{C}$-module. We say that $\F$ has depth $1$ if $\textrm{depth}(\F_{x})=1$ for all closed points $x\in \textrm{supp}(\F)$. It is called torsion free if it is Cohen-Macaulay and $\textrm{supp}(\F)=C$. It is of uniform multirank $r$ (or just rank $r$) if $\textrm{dim}(\F_{\mu})=r$ for every generic point of an irreducible component of $C$. 

Let $\Oo_{C}(1)$ be an ample invertible sheaf on the nodal curve $C$. The Hilbert polynomial of a coherent sheaf $\F$ over $C$ with respect to $\Oo_{C}(1)$ is 
$$P(\F,n):=h^{0}(C,\F(n))-h^{1}(C,\F(n)),$$
and the genus of $C$ is defined as $g:=1-\chi(\Oo_{C})$, $\chi(\Oo_{C}):=P(\Oo_{C},0)$ being its Euler characteristic.

A reduced projective and connected curve is said to be nodal if it has as much ordinary double points as singularities. Whenever we consider a nodal curve $C$, we will always be assuming that an ample invertible sheaf, $\Oo_{C}(1)$, has been fixed on it.

\subsection{Swamps and singular principal $G$-bundles}
Let $C$ be a nodal curve over $\mathbb{C}$, $P$ a polynomial with integral coefficients of degree one, and $\LL$ a locally free sheaf on $C$. We also fix natural numbers $a,b\in\mathbb{N}$.

A swamp (or tensor field) over $C$ is a pair $(\F,\phi)$ where $\F$ is a coherent $\Ox$-module of uniform multirank $r$, with Hilbert polynomial  $P$ and  a non-zero  morphism of $\Ox$-modules,
$\phi\colon (\F^{\otimes a})^{\oplus b}\rightarrow \LL$.
From now on we will say that $P$ is the Hilbert polynomial of the swamp $(\F,\phi)$ and $r$ is its rank.
Let $(\F,\phi)$ and $(\G,\gamma)$ be two swamps with the same Hilbert polynomial $P$. A morphism between them is a pair $(f,\alpha)$ where $\alpha\in \mathbb{C}$ and $f\colon \F\rightarrow\G$ is a morphism such that the following square commutes,
\begin{equation}\label{defiso}
\xymatrix{
(\F^{\otimes a})^{\oplus b}\ar[r]^{(f^{\otimes a})^{\oplus b}}\ar[d]^{\phi} &  (\G^{\otimes a})^{\oplus b}\ar[d]^{\gamma} \\
\LL\ar[r]^{\alpha \textrm{Id}} & \LL \ .
}
\end{equation}

Let $\F$ be a coherent $\Ox$-module on $C$. A weighted filtration, $(\F_{\bullet},\underline{m})$, of $\F$ is a filtration
\begin{equation*}
\F_{\bullet}\equiv (0)\subset\F_{1}\subset\F_{2}\subset\hdots\subset \F_{t}\subset\F_{t+1}=\F,
\end{equation*} 
equipped with positive numbers $\underline{m}=(m_{1}\hdots,m_{t})\in\mathbb{Q}^{t}_{>0}$. We adopt the following convention: the one step filtration is always equipped with $m=1$. A filtration is called saturated if the quotients $\F/\F_{i}$ are torsion free sheaves.

Consider now a swamp over $C$, $\phi\colon (\F^{\otimes a})^{\oplus b}\rightarrow \LL$,  and let $(\F_{\bullet},\underline{m})$ be a weighted filtration. For each $\F_{i}$, we denote by $\alpha_{i}$ its multiplicity  and just by $\alpha$ the multiplicity of $\F$. Define the vector 
$\Gamma=\sum_{1}^{t}m_{i}\Gamma^{(\alpha_{i})}$,
where 
$$\Gamma^{(l)}=(\overset{l}{\overbrace{l-\alpha,\hdots,l-\alpha}},\overset{\alpha-l}{\overbrace{l,\hdots,l}}).$$
Let us denote by $J$ the set
$J=\{\textrm{ multi-indices }I=(i_{1},\hdots,i_{a})|I_{j}\in\{1,\hdots,t+1\}\}.$
Define,
\begin{equation}
\mu(\F_{\bullet},\underline{m},\phi):= \textrm{min}_{I\in J}\{\Gamma_{\alpha_{i_{1}}}+\hdots +\Gamma_{\alpha_{i_{a}}}|\phi|_{(\F_{i_{1}}\otimes \hdots\otimes \F_{i_{a}})^{\oplus b}}\neq 0\}.
\end{equation}

\begin{definition}\label{semitensor}
Let $\delta\in\mathbb{Q}_{>0}$ be a positive rational number. A swamp $(\F,\phi)$ is $\delta$-(semi)stable if for each weighted filtration $(\F_{\bullet},\underline{m})$ the following holds
\begin{equation}\label{semiestabilidad}
\sum_{1}^{t}m_{i}(\alpha P_{\F_{i}}-\alpha_{i}P)+ \delta\mu(\F_{\bullet},\underline{m},\phi) (\leq)0.
\end{equation}
\end{definition}

Moduli spaces of swamps have been constructed under different assumptions in \cite{bosle,sols,L}. The most general result is given in \cite[Theorem 3.5]{L}, where A. Langer proves the existence of such a moduli space for relative schemes $X\rightarrow S$. The interesting part for us is the one regarding the one dimensional case.

\begin{theorem}\label{teolanger}
Let $P$ be a polynomial with integral coefficients of degree one, $f:\CC\rightarrow S$ a projective morphism of $k$-schemes of finite type with geometrically connected fibers of pure dimension one, and $\Oo_{\mathscr{C}}(1)$ a relative ample invertible sheaf. Let $\LL$ be a flat family of invertible sheaves on $\CC$ and $\delta\in\mathbb{Q}_{>0}$. There exists a coarse projective moduli space for $\delta$-semistable swamps $(\F,\phi)$ with Hilbert polynomial $P$.
\end{theorem}
\begin{proof}
This is part of \cite[Theorem 3.5]{L}.
\end{proof}

Let $C$ be a nodal curve over $\mathbb{C}$, $G$ a (semisimple) reductive group, $V$ a vector space of dimension $r$ and $\rho:G\hookrightarrow \textrm{SL}(V)$ a fully faithful representation.

\begin{definition}\label{maindef}
A singular principal $G$-bundle over $C$ is a pair $(\F,\tau)$ where $\F$ is a torsion free sheaf of rank $r$ and $\tau$ is a morphism of $\Oo_{\C}$-algebras,
$\tau:S^{\bullet}(V\otimes\F)^{G}\rightarrow\Oo_{\C},$
which is not just the projection onto the zero degree component.
\end{definition}
Singular principal $G$-bundles appear when we try to describe principal $G$-bundles as vector bundles with an extra structure (see \cite{Alexander2}). In fact, if we assume that $\F$ is locally free and the geometric counterpart of $\tau$, $\tau:C\rightarrow \underline{\textrm{Hom}}_{\Oo_{C}}(V\otimes\Oo_{C},\F^{\vee})\slas G$, takes values in the space of local isomorphisms, then we can recover the concept of principal $G$- bundle (see \cite{Alexander2}),
$$
\xymatrix{
\mathcal{P}(\F,\tau)\ar[r]\ar[d]_{\textrm{principal }\\ G \textrm{ bundle }} & \textrm{\underline{Isom}}_{\Ox}(V\otimes\Ox,\F^{\vee})\ar[d]^{\textrm{ principal }\\ G \textrm{ bundle}} \\
C\ar[r]^{\hspace{-2cm}\tau} &  \textrm{\underline{Isom}}_{\Ox}(V\otimes\Ox,\F^{\vee})\slas G \ .}
$$
Since considering only locally free sheaves leads to a non-compact moduli space, we need to add objects in order to get a compactification. Then, we define (see \cite{bosle, Alexander2, Alexander4} for the smooth and irreducible cases),
\begin{definition}
An honest singular principal $G$-bundle over $C$ is a singular principal $G$-bundle, $(\F,\tau)$, satisfying that $\tau:C\rightarrow \underline{\textrm{Hom}}_{\Oo_{C}}(V\otimes\Oo_{C},\F^{\vee})\slas G$ takes values in the space of local isomorphisms outside the nodes.
\end{definition}
One can define a semistability condition for honest singular principal $G$-bundles which, in case $C$ is smooth, leads to the same moduli space as the classical one constructed by Ramanathan (see \cite{Rama,Alexander2}).  

Since semistable honest singular principal $G$-bundles are difficult to deal with, we have to consider all singular principal $G$-bundles. We will show that we can assign (injectively up to isomorphisms) a swamp of certain type to every singular principal $G$-bundle, as happens in the irreducible case (see \cite{bosle, Alexander2}). Thus, we will say that a singular principal $G$-bundle is $\delta$-semistable if the associated swamp is. We will prove, in Theorem \ref{C}, that the type $(a,b)$ of the associated swamp does not depend on the base curve, but just on its genus. This let us to construct the relative moduli space for singular principal $G$-bundles. 

\section{\footnotesize EXTENDING THE LOCAL STRUCTURE OF TORSION FREE SHEAVES}\label{section3}
\markboth{\emph{}}{}

Let $C$ be a nodal curve, and let $\F$ be a torsion free sheaf on $C$ of rank $r$. C. S. Seshadri showed (see \cite[Chapter 8]{Seshadri}) that for each nodal point $x$ (regardless of how many components this point lies on), there is a natural number $0\leq a\leq r$ such that 
$
\F_{x}\simeq\mathscr{O}_{C,x}^{(r-a)}\oplus\mathfrak{m}_{x}^{a}.
$
This leads to the following definition.

\begin{definition}\label{structureconstants}
A torsion free sheaf of rank $r$ is of type $a_i$ at a nodal point $x_{i}$ if $\F_{x}\simeq \mathscr{O}_{C,x}^{(r-a_{i})}\oplus\mathfrak{m}_{x}^{a_i}$.
\end{definition}

\begin{lemma}\label{abierto}
If $\F$ is a torsion free sheaf on $C$ and $x$ is a node that lies in only one irreducible component $C'$ (respectively, that lies in two irreducible components $C'$, $C''$), then for each open subset $U$ such that the only node contained in it is $x$, and is contained in $C'$ (respectively contained in $C'\cup C''$), we have
$\F_{U}\subset \F_{x}$.
\end{lemma} 
\begin{proof}
Let $U$ be an open subset satisfying the above conditions. Consider the natural morphism $\Phi:\F_{U}\rightarrow\F_{x}$ that sends $s$ to $s/1$,
and let $s\in\F_{U}$ be an element of the kernel. Assume that $s\neq 0$.
Since $\dfrac{s}{1}=0$ there exists $0\neq f\in\mathscr{O}_{U}\setminus \overline{\mathfrak{m}}_{x}$ (the complement of the maximal ideal in $\mathscr{O}_{U}$ corresponding to the point $x$) such that $fs=0$. Consider the submodule $M:=<s>\subset\F_{U}$. Since $\Phi(s)=0$ we have $M_{x}=0$. Therefore there is a finite set of points $Z=\{p_{1},\hdots,p_{l}\}\subset U$ suth that $M_{V}=0$ with $V:=U\setminus Z$, that is, $M$ is supported on $Z$ (this is because of the properties of $U$). Then on each closed point of $Z$ we find that $M_{p_{i}}\simeq \mathbb{C}^{n_{i}}$, and therefore we have an inclusion $M_{p_{i}}\simeq \mathbb{C}^{n_{i}}\hookrightarrow\F_{p_{i}}$ for each $i$ which can not be possible because $\textrm{depth}(\F_{p_{i}})=1$, so $s=0$ and $\Phi$ is injective.
\end{proof}

\begin{theorem}\label{localstructure}
Let $\F$ be a torsion free sheaf on $C$ of rank $r$. Let $x\in C$ be a node and suppose that $\F$ is of type $a$ on $x$, i.e. $\F_{x}\simeq \mathscr{O}_{C,x}^{r-a}\oplus \mathfrak{m}_{x}^{a}$. Then there is an (affine) open neighborhood, $U$, of $x$ not containing more nodes and an isomorphism
\begin{equation*}
\Phi_{U}:\F_{U}\simeq \mathscr{O}_{U}^{r-a}\oplus \overline{\mathfrak{m}}_{x}^{a}
\end{equation*}
satisfying $\Phi_{U,x}=\Phi_{x}$.
\end{theorem}
\begin{proof}
Let $V$ be an open subset as in the Lemma \ref{abierto}. Then we have
\begin{equation*}
\xymatrix{
\mathcal{F}_{x} \ar[r]^{\Phi_{x}} &  \mathscr{O}_{C,x}^{r-a} \oplus \mathfrak{m}_{x}^{a} \\
\mathcal{F}_{V}\ar@{^(->}[u] &
}
\end{equation*}
Let $\{m_{1},\hdots,m_{l}\}$ be  generators of $\F_{V}$. Then
\begin{equation*}
\Phi_{x}(m_{i})=\dfrac{s_{i}}{t_{i}}+\dfrac{f_{i}}{g_{i}}\textit{, }s_{i}\in\mathfrak{m}_{x}^{a}\textit{, } f_{i}\in\mathscr{O}_{C,x}^{r-a}\textit{, }g_{i},t_{i}\in \mathscr{O}_{V}\setminus\overline{\mathfrak{m}}_{x}.
\end{equation*}
Consider the ideal $I=<\{g_{i}\},\{t_{i}\}>$ and let $U'=V\setminus V(I)$. Then we have a commutative diagram
$$
\xymatrix{
\F_{x}\ar[r]^{\Phi_{x}} &\mathscr{O}_{C,x}^{r-a} \oplus  \mathfrak{m}_{x}^{a} \\
\F_{V}\ar@{^(-->}[r]^{\Phi_{V}}\ar@{^(->}[u] &  \mathscr{O}_{U'}^{r-a} \oplus \overline{\mathfrak{m}}_{x}^{a}   \ar@{^(->}[u]
}
$$
Since $\Phi_{U',x}$ is an isomorphism there exists an open subset $U'\supseteq U\ni x$ such that $\Phi_{U}$ is an isomorphism.
\end{proof}

\section{\footnotesize SOME RESULTS ON GRADED ALGEBRAS}\label{section4}
\markboth{\emph{}}{} 

e are going to prove that  given a submodule of a commutative $R$-algebra generating it at a point $\mathfrak{p}_{x}\subset R$, then the submodule generates it over some open neighborhood of that point. This result will be important in order to get a satisfactory linearization of our moduli problem for principal bundles.
\begin{lemma}\label{generatingalgebras}
Let $R$ be a commutative ring and $X=\emph{Spec}(R)$. Let $B$ be a finitely generated commutative $R$-algebra and $A\subset B$ a sub-$R$-algebra. Let $x\in X$ be a point such that $A_{x}=B_{x}$, then there exists an affine open neighborhood $U\subset X$ of $x$ such that $A_{U}=B_{U}$.
\end{lemma}
\begin{proof}
Consider such a point $x$. From the equality $A_{x}=B_{x}$ it follows that $\dfrac{b}{s}\in A_{x}$ for any $b\in B$ and any $\ s\in R-\mathfrak{p}_{x}$, that is, for all $b\in B$ there exists $s\in R-\mathfrak{p}_{x}$ such that $sb\in A$. Since $B$ is finitely generated we can consider $\{b_{1},\hdots, b_{l}\}$ a finite set of generators. For any $i=1,\hdots, l$ there exists $s_{i}\in R-\mathfrak{p}_{x}$ such that $b_{i}s_{i}\in A$. Define $f:=s_{1}\cdot\hdots\cdot s_{l}$. Let us show now that for any $b\in B$ there exists a natural number $n\in\mathbb{N}$ such that $f^{n}b\in A$. Let $b\in B$. Since $B$ is finitely generated there exists a polynomial $P\in R[x_{1},\hdots,x_{l}]$ such that 
$b=P(b_{1},\hdots,b_{l})=\sum a_{j_{1}\hdots j_{l}}b_{1}^{k_{1}}\hdots b_{l}^{k_{l}}.$
Denote $n:=\textrm{deg}(P)$. Then
$$
f^{n}b=\sum f^{d_{i,\hdots,l}}a_{j_{1}\hdots j_{l}}(fb_{1})^{k_{1}}\hdots (fb_{l})^{k_{l}}, \ \textrm{where }d_{1,\hdots,l}=n-k_{1}-\hdots-k_{l}.
$$
Since all the terms in this sum belong to $A$ it follows that $f^{n}b\in A$. Consider now the inclusion $A_{f}\subset B_{f}$ and an element $b/f^{e}\in B_{f}$. From the last argument there is a natural number $n\in\mathbb{N}$ such that $b/f^{e}=bf^{n}/f^{e+n}\in A_{f}$,
from what we deduce that $A_{f}=B_{f}$.
\end{proof}

\begin{theorem}\label{gen}
Let $C$ be a nodal curve over $\mathbb{C}$ with nodes $z_{1},\hdots,z_{\nu}$, $G$ a reductive group and fix a representation $\rho: G\hookrightarrow \emph{SL}(V)$, $V$ being a $\mathbb{C}$-vector space of dimension $p$. Fix $r\in\mathbb{N}$. Then there exists a natural number $s=s(p,r,\nu)$ such that for any torsion free sheaf $\F$ of rank $r$ on $C$, the graded $\Ox$-algebra $S^{\bullet}(V\otimes\F)^{G}$ is generated by the submodule $\bigoplus_{i=0}^{s}S^{i}(V\otimes\F)^{G}$.
\end{theorem}
\begin{proof}
Let us denote $\B=S^{\bullet}(V\otimes\F)^{G}$. We will find this natural number in three steps.

Step 1:  a) Let $\F$ be a torsion free sheaf on $C$ of rank $r$ and structure constants $a_{1},\hdots,a_{\nu}$ (see Definition \ref{structureconstants}). Let $x\in C$ be a point and consider the $\mathscr{O}_{C,x}$-algebra $\B_{x}=S^{\bullet}(V\otimes\F_{x})^{G}$. Let $s_{x}$ be the minimal natural number such that $\bigoplus_{i=1}^{s_{x}}S^{i}(V\otimes\F_{x})^{G}$ contains a set of generators of $\B_{x}$.

b) For any point $x\in C$ we construct an affine open neighborhood in the following way. Fix $x\in C$ and consider the sub-$\Ox$-algebra $\mathcal{A}\subset \B$ generated by the sub-$\Ox$-module $\bigoplus_{i=0}^{s_{x}}(V\otimes\F)^{G}$. Then obviously $\mathcal{A}_{x}=\B_{x}$. By Theorem \ref{generatingalgebras} we deduce that there exists an affine open neighborhood $U_{x}$ such that $\mathcal{A}_{U_{x}}=\B_{U_{x}}$.

c) We get in this way a covering of $C$, $\{U_{x}\}_{x\in C}$ by affine open subschemes and we can choose finitely many regular points $x_{1},\hdots,x_{l}\in C$ such that $C=U_{z_{1}}\cup\hdots\cup U_{z_{\nu}}\cup U_{x_{1}}\cup\hdots\cup U_{x_{l}}$. Let $s_{z_1},\hdots,s_{z_{\nu}},s_{x_1},\hdots,s_{x_l}$ be the corresponding natural numbers defined in a) and define $s'=\textrm{max}(s_{z_1},\hdots,s_{z_{\nu}},s_{x_1},\hdots,s_{x_l})$.

d) Note that the natural number $s'$ constructed in c) does not depend on the finite open cover we have chosen, since $s_{x_{1}}=\cdots=s_{x_{l}}$ for all regular points.

Step 2:  The natural number constructed in Step 1 depends apparently on $\F$, but it does not. Actually, it just depends on the structure constants $a_{1},\hdots,a_{\nu}$, the rank $r$ and $p$. Suppose we have two torsion free sheaves $\F, \ \G$ of rank $r$ and the same structure constants. Note that the natural numbers defined in a) depends just on $\F_{x}$ and $\G_{x}$. But with the above assumption, $\F_{x}\simeq\G_{x}$ for all $x\in \C$, and therefore  $s'=s'(r,p,a_{1},\hdots,a_{\nu})$.

Step 3: Since there are only finitely many possibilities for the structure constants (once we fix the rank) we get in this way finitely many natural numbers $s'$. Consider the maximum of all of them and denote it by $s$. Then, $s$ depends just on $p, r$ and $\nu$, and satisfies the properties of the statement.

\end{proof}
\begin{remark}\label{genuniforme}

1) The last theorem also holds for non-connected nodal curves.

2) Let $\pi:Y\rightarrow C$ be the normalization of $C$. If $C$ has $\nu$ nodal points, then we have an exact sequence,
\begin{equation*}
0\rightarrow\Ox\rightarrow\pi_{*}\Oy\rightarrow\bigoplus_{i=1}^{\nu}\mathbb{C}\rightarrow 0,
\end{equation*}
and therefore $g_{\C}=g_{Y}+\nu$, $g_{C}$ and $g_{Y}$ being the genus of $C$ and $Y$ respectively. Since both,  $g_{Y}$ and $\nu$, are natural numbers, we deduce that for a fixed genus $g=g_{C}$ there are finitely many possibilities for the number of nodes $\nu$. Fixing the genus $g=g_{C}$ and taking the maximum of the numbers $s=s(p,r,\nu)$ varying $\nu$, we get a number $s=s(p,r,g)$ which does not depend on the curve $C$. In our particular case $p=r$, thus $s=s(r,g)$.
\end{remark}

\section{\footnotesize  UNIFORM LINEARIZATIONS OF SINGULAR PRINCIPAL \\ $G$-BUNDLES } \label{section5}
\markboth{\emph{}}{} 

Let $C$ be a (possibly non-connected) nodal curve. Consider a singular principal $G$-bundle on $C$,
$\tau:S^{\bullet}(V\otimes\F)^{G}\rightarrow\Ox$.
Let $s\in\mathbb{N}$ be as in the last section. Then $S^{\bullet}(V\otimes\F)^{G}$ is generated by the submodule $\bigoplus_{i=0}^{s}S^{i}(V\otimes \F)^{G}$. Let $\underline{d}\in\mathbb{N}^{s}$ be such that $\sum id_{i}=s!$. Then we have:
\begin{equation}\label{asten}
\bigotimes_{i=1}^{s}(V\otimes\F)^{\otimes id_{i}} \rightarrow\bigotimes_{i=1}^{s}S^{d_{i}}(S^{i}(V\otimes\F))\rightarrow \bigotimes_{i=1}^{s}S^{d_{i}}(S^{i}(V\otimes\F))^{G}\rightarrow\Ox
\end{equation}
Adding up these morphisms as $\underline{d}\in\mathbb{N}$ varies we find a swamp (see \cite{Alexander2})
\begin{equation}\label{associatedtensor}
\varphi_{\tau}:((V\otimes\F)^{\otimes s!})^{\oplus N}\rightarrow\Ox.
\end{equation}
We want to prove that the linearization map $\tau\mapsto\varphi_{\tau}$ 
is injective up to isomorphisms. Let us start with the following proposition.

\begin{proposition}\label{root}
Let $x\in C$ be a node, $A=\mathscr{O}_{C,x}$ the local ring and $\tau,\tau':A^{m}\rightarrow A$ two non zero morphisms such that $S^{d}(\tau)=S^{d}(\tau')$ for some natural number $d\in\mathbb{N}$. Then there exists a $d$th root of unity $\xi$ such that $\tau'=\xi\tau$.
\end{proposition}
\begin{proof}
We have to consider two cases.

a) The point $x$ is the intersecting point of two components: Let $\mathfrak{p}_{\eta_{1}},\mathfrak{p}_{\eta_{2}}$ be the minimal ideals of $A$.
Let $\{e_{1},\hdots,e_{m}\}$ be the canonical basis of $A^{m}$. Then $S^{d}(\tau)=S^{d}(\tau')$ means that
\begin{equation}\label{eqp}
\tau(e_{i_{1}})\hdots\tau(e_{i_{d}})=\tau'(e_{i_{1}}) \hdots \tau'(e_{i_{d}}), \ \forall (i_{1},\hdots,i_{d})\textrm{ with }1\leq i_{j}\leq m.
\end{equation}
In particular $\tau(e_{j})^{d}=\tau'(e_{j})^{d}$, $\forall \ 1\leq j\leq m$. Let $j$ be such that $\tau(e_{j})$ and $\tau'(e_{j})$ are non zero. Since $A$ is reduced this means that
$\tau(e_{j}),\tau'(e_{j})\not\in \mathfrak{p}_{\eta_{1}}\cap \mathfrak{p}_{\eta_{2}}=(0)$
so, in particular, there exists $l=1,2$ such that
$\tau(e_{j}),\tau'(e_{j})\not\in\mathfrak{p}_{\eta_{l}}$,
and therefore $\tau(e_{j}),\tau'(e_{j})$ are invertible in $A_{\mathfrak{p}_{\eta_{l}}}=\Sigma_{l}$. Consider the equality $\tau(e_{j})^{d}=\tau'(e_{j})^{d}$ in $\Sigma_{l}$. This exactly means that
\begin{equation*}
(\dfrac{\tau(e_{j})}{\tau'(e_{j})})^{d}=1\textrm{ in }\Sigma_{l},
\end{equation*}
so there exists a $d$th root of unity $\xi_{j}$ such that $\tau'(e_{j})=\xi_{j}\tau(e_{j})$ for all $j$ with $\tau(e_{j}),\tau'(e_{j})\neq 0$.
From Equation (\ref{eqp}), we deduce that $\xi_{j}$ does not depend on $j$, so $\tau=\xi\tau'$.

b) The point $x$ is not an intersecting point of two components: This case follow from the last part of the above argument.
\end{proof}

\begin{corollary}\label{chi}
Let $x\in C$ be a node, $A=\mathscr{O}_{C,x}$ the local ring, $M$  a finitely generated $A$-module and $\tau,\tau':M\rightarrow A$ two non zero morphisms such that $S^{d}(\tau)=S^{d}(\tau')$ for some natural number $d\in\mathbb{N}$. Then there exists a $d$\textrm{th} root of unity $\xi$ such that $\tau'=\xi\tau'$.
\end{corollary}
\begin{proof}
Let $\{m_{1},\hdots,m_{t}\}\in M$ be such that $\{\overline{m}_{1},\hdots,\overline{m}_{t}\}\in M\otimes_{A}\mathbb{C}$ is a basis. Consider the canonical surjection
\begin{align*}
\pi:A^{t}&\rightarrow M .\\
e_{j}&\mapsto m_{j}
\end{align*}
Composing with $\tau$ and $\tau'$ we find two morphisms
$\tau\circ\pi,\tau'\circ\pi:A^{t}\rightarrow A$.
By Theorem \ref{root} there exists a $d$th root of unit $\xi$ such that $\tau'\circ\pi=\xi\tau\circ\pi$, that is
$\tau'(m_{j})=\xi\tau(m_{j})\textrm{ for all } 1\leq j\leq t$,
so $\tau'=\xi\tau$.
\end{proof}
 From now onwards, until the end of this section, we will denote by $C_{1},\hdots,C_{t}$ the connected components of $C$.
\begin{theorem}\label{xiglobal}
Let $\F$ be a coherent $\Ox$-module on $C$ with $\emph{supp}(\F)=C$ and  $\tau,\tau':\F\rightarrow\Ox$ non zero morphisms such that $S^{d}(\tau)=S^{d}(\tau')$ with $d\in\mathbb{N}$. Then there exists a $d$\textrm{th} root of unity $\xi_{i}$, one for each connected component, such that 
\begin{equation*}
\tau'|_{C_{i}}=\xi_{i}\tau|_{C_{i}}, \ i=1,\hdots, t.
\end{equation*}
\end{theorem}
\begin{proof}
For all $x\in C$ we have $\tau,\tau':\F\rightarrow\mathscr{O}_{C,x}$ with $S^{d}(\tau_{x})=S^{d}(\tau'_{x})$. By Corollary \ref{chi} there is a $d$th root of unity $\xi_{x}$ such that $\tau'_{x}=\xi_{x}\tau_{x}$. We know that for any point $x$ there is an open subset $x\in U\subset C$ such that $(\F^{\vee})_{U}\subset(\F^{\vee})_{x}$. This is because the dual of a coherent sheaf with $\textrm{supp}(\F)=C$ always has $\textrm{depth}=1$ and we can apply Lemma \ref{abierto}. Thus, for any point $x\in C$ there is an open subset $U$ such that
\begin{equation*}
\tau'_{U}=\xi_{x}\tau_{U},
\end{equation*}
from which we deduce that $\xi_{x}$ does not depend on $x$ and that $\tau'|_{C_{i}}=\xi_{i}\tau|_{C_{i}}$.
\end{proof}

\begin{lemma}\label{root}
Let $s\in\mathbb{N}$ and $\xi_{i}$ a $\dfrac{s!}{i}$\textrm{th} root of unity for $i=1,\hdots s$ such that for any partition of $s!$,
$d_{1}+2d_{2}+\hdots sd_{s}=s!$
the following holds: 
\begin{equation}\label{igualdad}
1=\prod_{i=1}^{s}\xi_{i}^{d_{i}}.
\end{equation}
Then $\xi_{1}^{j}=\xi_{j}$ for all $j\in\{1,\hdots,s\}$.
\end{lemma}
\begin{proof}
Consider the complex representation for each $\xi$
$\xi_{j}=\textrm{exp}(2\pi i\dfrac{jk_{j}}{s!}), \ k_{j}\in\mathbb{N}.$
Equations (\ref{igualdad}) are equivalent to
\begin{equation}\label{igualdad2}
k_{1}d_{1}+2k_{2}d_{2}+\hdots + sk_{s}d_{s}=0\textrm{ mod}(s!)\textrm{ for all partitions } (d_{1},\hdots,d_{s})
\end{equation}
which written in matrix form become to
\begin{equation*}
\left(\begin{array}{cccc}
d_{1}^{1} & 2d_{2}^{1} & \hdots & sd_{s}^{1} \\
d_{1}^{2} & 2d_{2}^{2} & \hdots & sd_{s}^{2} \\
\vdots & \vdots & \vdots & \vdots \\
d_{1}^{m} & 2d_{2}^{m} & \hdots & sd_{s}^{m}
\end{array}\right)\left( \begin{array}{c}
k_{1} \\
k_{2} \\
\vdots \\
k_{s}
\end{array}\right)= \left(\begin{array}{c}
0 \\
0 \\
\vdots \\
0
\end{array}\right)\ \ \ \textrm{mod}(s!).
\end{equation*}
Note that we can easily find $s$ linearly independent solutions for this linear system
\begin{equation*}
\left(\begin{array}{c}
1 \\
1 \\
\vdots \\
1
\end{array}\right), \ \ \left(\begin{array}{c}
0 \\
\dfrac{s!}{2} \\
\vdots \\
0 
\end{array}\right), \ \ \left(\begin{array}{c}
0 \\
0 \\
\vdots \\
\dfrac{s!}{s}
\end{array}\right).
\end{equation*}
Therefore the general solution for the system is given by
\begin{equation}\label{igualdadgorda}
\left(\begin{array}{c}
k_{1} \\
k_{2} \\
\vdots \\
k_{s}
\end{array}\right)= \lambda_{1}\left(\begin{array}{c}
1 \\
1 \\
\vdots \\
1
\end{array}\right)+\lambda_{2}\left(\begin{array}{c}
0 \\
\dfrac{s!}{2} \\
\vdots \\
0 
\end{array}\right)+\hdots +\lambda_{s} \left(\begin{array}{c}
0 \\
0 \\
\vdots \\
\dfrac{s!}{s}
\end{array}\right)= \left(\begin{array}{c}
\lambda_{1} \\
\lambda_{1}+ \lambda_{2}\dfrac{s!}{2} \\
\vdots \\
\lambda_{1}+\lambda_{s}\dfrac{s!}{s}
\end{array}\right).
\end{equation}
Note also that $\xi_{1}^{j}=\xi_{j}$ if and only if $jk_{1}-jk_{j}=0\textrm{ mod}(s!)$. But by Equation (\ref{igualdadgorda}), we know that
$jk_{1}-jk_{j}=jk_{1}-j(\lambda_{1}+\lambda_{j}\dfrac{s!}{j})=-\lambda_{j}s!=0 \textrm{ mod}(s!).$
\end{proof}

\begin{theorem}\label{C}
Let $G$ be a reductive algebraic group and let $\rho:G\hookrightarrow \emph{SL}(V)$ be a fully faithful representation, $V$ being a $\mathbb{C}$-vector space of dimension $r$. Let $C$ be a (possibly non-connected) nodal curve over $\mathbb{C}$ of genus $g$. Then there are parameters $(a,b)$, depending only on $r$ and $g$, such that the mapping, 
\begin{equation*}
\left\{   \begin{array}{c}  \textit{isomorphism classes} \\ \textit{of singular principal} \\ \textit{G-bundles}  \end{array} \right\} \rightarrow \left\{   \begin{array}{c}  \textit{isomorphism classes} \\ \textit{of swamps} \\ \textit{of type }(a, b, \Ox)   \end{array} \right\} 
\end{equation*}
is injective.
\end{theorem}
\begin{proof}
Define $a:=s!$ as in Theorem \ref{gen} and Remark \ref{genuniforme}. Now, $b:=N$ is obtained from $a$ as we have seen in the last section. Following \cite{Asch} pp 187, we have to show that, if $\phi_{\tau}=\phi_{\tau'}$ then $(\F,\tau)\simeq(\F,\tau')$. For $i>0$ consider
\begin{equation*}
\tau_{i}, \ \tau'_{i}:S^{i}(V\otimes\F)^{G}\rightarrow\Ox,
\end{equation*}
the degree $i$ components of $\tau$ and $\tau'$. Observe that $\tau$ (respectively $\tau'$) is completely determined by $\oplus_{i=1}^{s}\tau_{i}$ (resp. $\oplus_{i=1}^{s}\tau_{i}$). Consider now:
\begin{equation*}
\widehat{\tau_{{s}}}:\bigoplus_{\tiny \begin{array}{c} (d_{1},\hdots,d_{s}) \\ \sum id_{i}=s! \end{array}}(S^{d_{1}}((V\otimes\F)^{G})\otimes \hdots \otimes S^{d_{s}}(S^{s}(V\otimes\F)^{G}))\rightarrow \Ox,
\end{equation*}
the morphism induced by $\tau_{1},\hdots,\tau_{s}$. Define in the same way $\widehat{\tau'_{s}}$. Note that from the surjectivity of the first two morphisms defining (\ref{asten}) it follows that if $\varphi_{\tau}=\varphi_{\tau'}$ then $\widehat{\tau_{s}}=\widehat{\tau'_{s}}$. This implies, in particular, that
\begin{equation*}
S^{\begin{tiny}s!/i\end{tiny}} (\tau_{i})=S^{\tiny s!/i}(\tau'_{i}), \ \forall \ 0<i\leq s.
\end{equation*}
From Theorem \ref{xiglobal} and Lemma \ref{root} we deduce that for any connected component $C_{j}$ ($j=1,\hdots, t$) there exists an $s!$th root of unity $\xi_{j}$ such that:
\begin{equation*}
\tau_{i}'|_{C_{j}}=(\xi_{j})^{i}\tau_{i}|_{C_{j}}, \ i=1,\hdots, s.
\end{equation*}
Denote by  $u=\textrm{diag}(\xi_{1},\hdots,\xi_{t}):\F\simeq \F$
the induced automorphism on $\F$. If we apply $u$ to the singular principal $G$-bundle $(\F,\tau)$ we get a singular principal $G$-bundle $(\F,\tau'')$,
$$
\xymatrix{
S^{\bullet}(V\otimes\F)^{G}\ar[r]^{\tiny S^{\bullet}(1\times u)}\ar@{-->}[dr]^{\tau''} & S^{\bullet}(V\otimes\F)^{G}\ar[d]^{\tau} \\
& \Ox
}
$$
Clearly, $\tau_{i}''=\tau_{i}', \ \forall \ 0<i\leq s$ on each connected component, and therefore on the whole curve. Since $s$ is large enough we deduce that $\tau'=\tau''$, hence $(\F,\tau)\simeq (\F,\tau'')$.
\end{proof}

\section{\footnotesize THE MODULI SPACE OF SINGULAR PRINCIPAL $G$-BUNDLES} \label{section6}
\markboth{\emph{}}{} 

Let $C$ be a nodal curve, $\delta\in\mathbb{Q}_{>0}$ and $P$ a polynomial of degree one with integral coefficients. Let $\tau\colon S^{\bullet}(V\otimes\F)^{G}\rightarrow\Ox$ be a singular principal $G$-bundle. Let $s\in\mathbb{N}$ be as in Theorem \ref{gen}, so $\bigoplus_{i=0}^{s}S^{i}(V\otimes\F)^{G}$ contains a set of generators, and let $\phi_{\tau}\colon ((V\otimes\F)^{\otimes s!})^{\oplus N}\rightarrow\Ox$ be the associated swamp. 
\begin{definition} \label{semistability00}
Let $\delta\in\mathbb{Q}_{>0}$ be a positive rational number . A singular principal $G$-bundle $(\F,\tau)$ is $\delta$-(semi)stable if its associated swamp $(\F,\phi_{\tau})$ is $\delta$-(semi)stable.
\end{definition}

The aim of this section is to prove the existence of a coarse projective moduli space for the moduli functor given by
$$
\textbf{SPB}(\rho)_{P}^{\delta\textrm{-(s)s}}(S)=\left\{   \begin{array}{c}\textrm{isomorphism classes of}\\ \textrm{families of } \delta\textrm{-(semi)stable} \textrm{ singular} \\ \textrm{principal G-bundles on }C \textrm{ parametrized }\\ \textrm{by S with Hilbert polynomial }P \end{array}  \right\}.
$$
In order to do so, we rigidify the functor to get the parameter space and then we take the quotient of the parameter space by the group of isomorphisms of the rigidification. Observe that once Theorem \ref{C} is proven, the construction of the moduli space is done following the same strategy as in the smooth case \cite{Alexander2}. Therefore we will just describe the main steps.
\subsection{The Parameter Space.}

To rigidify the moduli problem we fix $n\in\mathbb{N}$ very large, as in the proof of \cite[Theorem 6.1]{L}, and $W$ a vector space of dimension $P(n)$. Consider the functor 
\begin{equation*}
^{\textrm{rig}}\textbf{SPB}(\rho)_{P}^{n}(S)=\left\{   \begin{array}{c} \textrm{isomorphism classes of tuples }(\F_{S},\tau_{S},g_{S}) \textrm{ where } \\
 (\F_{S},\tau_{S}) \textrm{ is a family of singular principal G-bundles} \\ \textrm{parametrized by }S 
  \textrm{ with Hilbert polynomial P and } \\
g_{S}\colon W\otimes\mathscr{O}_{S}\rightarrow\pi_{S*}\F_{S}(n)\textrm{ is a morphism such that the}\\
\textrm{induced morphism }W\otimes\pi^{*}\Ox\rightarrow \F_{S}(n)\textrm{ is surjective}
 \end{array}\right\},
\end{equation*}
and let us show that there is a representative for it. Let $\mathcal{Q}$ be the Quot scheme of quotients,
$q: W\otimes\Ox(-n)\rightarrow\F$,
with Hilbert polynomial $P$. Consider the following morphism on $\mathcal{Q}\times C$
\begin{equation*}
h\colon S^{\bullet}(V\otimes W\otimes\pi^{*}_{C}\Ox(-n))\twoheadrightarrow S^{\bullet}(V\otimes\F_{\mathcal{Q}}) \twoheadrightarrow S^{\bullet}(V\otimes\F_{\mathcal{Q}})^{G},
\end{equation*}
induced by the universal quotient and the Reynolds operator. Let $s\in\mathbb{N}$ be as in Theorem \ref{gen} and Remark \ref{genuniforme}. Then
$
h(\bigoplus_{i=1}^{s}S^{i}(V\otimes W\otimes\pi_{C}\Ox(-n)))
$
contains a set of generators of $S^{\bullet}(V\otimes\F_{\mathcal{Q}})^{G}$. A morphism $k\colon \bigoplus_{i=1}^{s}S^{i}(V\otimes W\otimes\Ox(-n))\rightarrow\Ox$ breaks into a family of morphisms
$$
k^{i}\colon S^{i}(V\otimes W)\otimes\Ox(-in)\simeq S^{i}(V\otimes W\otimes\Ox(-n))\rightarrow\Ox,
$$
obtaining, therefore, a family of linear maps
$
k^{i}\colon S^{i}(V\otimes W)\rightarrow H^{0}(\Ox(in)).
$
Thus, any singular principal $G$-bundle $\tau\colon S^{\bullet}(V\otimes\F)^{G}\rightarrow\Ox$ is completely determined by a point in the space
\begin{equation*}
\mathcal{Q}^{*}:=\mathcal{Q}\times\bigoplus_{i=1}^{s} \textrm{\underline{Hom}}(S^{i}(V\otimes W),H^{0}(\Ox(in))).
\end{equation*}
We want to put a scheme structure on the locus given by the points $([q],[k])\in\mathcal{Q}^{*}$ that comes from a morphism of algebras
$
S^{\bullet}(V\otimes\F_{\mathcal{Q}^{0}|_{[q]}\times C})^{G}\rightarrow\Ox.
$
On $\mathcal{Q}^{*}\times C$ there are universal morphisms
\begin{equation*}
\varphi'^{i}\colon S^{i}(V\otimes W)\otimes\mathcal{O}_{\mathcal{Q}^{*}\times C}\rightarrow H^{0}(\Ox(in))\otimes\mathcal{O}_{\mathcal{Q}^{*}\times C}
\end{equation*}
Consider the pullbacks of the evaluation maps to $\mathcal{Q}^{*}\times C$
\begin{equation*}
H^{0}(\Ox(in))\otimes\mathcal{O}_{\mathcal{Q}^{*}\times C}\rightarrow\pi_{C}^{*}\Ox(in).
\end{equation*}
Composing both we get
$\varphi^{i}\colon S^{i}(V\otimes W)\otimes \mathcal{O}_{\mathcal{Q}^{*}\times C}\rightarrow\pi_{C}^{*}\Ox(in)$,
and summing up over all $i$ we find
$\varphi\colon V_{\mathcal{Q}^{*}}\colon =\bigoplus_{i=1}^{s}S^{i}(V\otimes W\otimes\pi_{C}^{*}\Ox(-n))\rightarrow\mathcal{O}_ {\mathcal{Q}^{*}\times C}$.
Now $\varphi$ gives a morphism
$\tau'_{\mathcal{Q}^{*}}\colon S^{\bullet}(V_{\mathcal{Q}^{*}})\rightarrow\mathcal{O}_{\mathcal{Q}^{*} \times C}.$
Consider again the universal quotient $q_{\mathcal{Q}}$ and the following chain of surjections
$$
\xymatrix{
S^{\bullet}(V\otimes W\otimes\otimes\pi_{C}^{*}\Ox(-n))\ar@{->>}[rr]^{S^{\bullet}(1\otimes q_{\mathcal{Q}})} & & S^{\bullet}(V\otimes\pi^{*}_{\mathcal{Q}^{*}\times C}\F_{\mathcal{Q}^{*}})\ar@{->>}[d]^-{\textrm{Reynolds}} \\
S^{\bullet}V_{\mathcal{Q}^{*}}\ar@{->>}[u] & & S^{\bullet}(V\otimes\pi^{*}_{\mathcal{Q}^{*}\times C}\F_{\mathcal{Q}^{*}})^{G}
}
$$
Let us denote by $\beta$ the composition of these morphisms and consider the diagram
$$
\xymatrix{
0\ar[r] & \textrm{Ker}(\beta)\ar@{^(->}[r]\ar@{-->}[rd]^{\tau'_{\mathcal{Q}^{*}}} & S^{\bullet}V_{\mathcal{Q}^{*}}\ar[r]^{\hspace{-0,5cm}\beta} \ar[d]^{\tau'_{\mathcal{Q}^{*}}} & S^{\bullet}(V\otimes\pi_{\mathcal{Q}^{*}\times C}^{*}\F_{\mathcal{Q}^{*}})^{G} \ar[r] & 0 \\
 & & \mathscr{O}_{\mathcal{Q}^{*}\times C} & &
}
$$
Define $\mathbb{D}=\{c=([q],[h])|\tau'_{\mathcal{Q}^{*}}|_{c}=0\}$. This is a closed subscheme of $\mathcal{Q}^{*}$ over which $\tau'_{\mathcal{Q}^{*}}$ lifts to
$\tau_{\mathbb{D}}\colon  S^{\bullet}(V\otimes\pi_{\mathcal{Q}^{*}\times C}^{*}\F_{\mathcal{Q}^{*}})^{G}\rightarrow \mathscr{O}_{\mathcal{Q}^{*}\times C}$.
To see this, note that $\mathbb{D}=\cap_{d\geq 0}\mathbb{D}^{d}$ with
$\mathbb{D}^{d}:=\{c=([q],[h])|  \tau'^{d}_{\mathcal{Q}^{*}}|_{c}:\textrm{Ker}(\beta^{d})|_{c}\rightarrow\Ox\textrm{ is trivial}\}$
which, by \cite[Lemma 3.1]{sols}, are closed in $\mathcal{Q}^{*}$.  Then, we have,

\begin{theorem}
The functor $^{rig}\emph{\textbf{SPB}}(\rho)^{n}_{P}$ is represented by the scheme $\mathbb{D}$.
\end{theorem}
\begin{proof}
Follows from the construction of $\mathbb{D}$.
\end{proof}

\subsection{Construction of the moduli space.}

Recall, from \cite[Theorem 6.1]{L}, that the family of torsion free sheaves $\F$ appearing in $\delta$-(semi)stable swamps is bounded.
As a consequence,  there is a natural number $n\in\mathbb{N}$ large enough such that $\F(n)$ is globally generated and $h^{1}(C,\F(n))=0$. Fix such a natural number $n$, $a$ and $b$ as in Theorem \ref{C}, and consider the functors
\begin{align*}
^{\textrm{rig}}\textbf{Swamps}^{n}_{P,\mathscr{O}_{C}}(S)&=\left\{
\begin{array}{c}
\textrm{isomorphism classes of tuples }(\F_{S},\phi_{S},N,g_{S})\textrm{ where}\\
\F_{S} \textrm{ is a coherent sheaf with Hilbert polynomial }P,\\
\textrm{over }C\times S \textrm{ flat over }S, g_{S}\colon W\otimes\mathscr{O}_{S}\rightarrow \pi_{S*}\F_{S}(n) \\
\textrm{ is such that its image generates } \F_{S}\textrm{ and}\\
\phi\colon ((V\otimes\F_{S})^{\otimes a})^{\oplus b}\rightarrow \pi_{S}^{*}N\textrm{ is a morphism}\\
\end{array}\right\}, \\
^{\textrm{rig}}\textbf{SPB}(\rho)_{P}^{n}(S)&=\left\{   \begin{array}{c} \textrm{isomorphism classes of tuples }(\F_{S},\tau_{S},g_{S}) \textrm{ where } \\
 (\F_{S},\tau_{S}) \textrm{ is a family of singular principal G-bundles} \\ \textrm{parametrized by }S 
  \textrm{ with Hilbert polynomial P and } \\
g_{S}\colon W\otimes\mathscr{O}_{S}\rightarrow\pi_{S*}\F_{S}(n)\textrm{ is a morphism such that the}\\
\textrm{induced morphism } W\otimes\pi^{*}\Ox\rightarrow\F_{S}(m)\textrm{ is surjective}
 \end{array}\right\}.
\end{align*}

Note that the natural $\textrm{GL}(W)$-action on the universal quotient $W\otimes\pi_{C}^{*}\Ox(-n)\twoheadrightarrow \F_{\mathcal{Q}^{*}}$ determines an action on the space $\mathbb{D}$,
$$\Gamma\colon \textrm{GL}(W)\times\mathbb{D}\rightarrow\mathbb{D}.$$
We can view the $\textrm{GL}(W)$-action as a $\mathbb{C}^{*}\times \textrm{SL}(W)$-action. Thus, we can construct the quotient of $\mathbb{D}$ by $\textrm{GL}(W)$ in two steps, considering the actions of $\mathbb{C}^{*}$ and $\textrm{SL}(W)$ separately. 
Consider the  action of $\mathbb{C}^{*}$ on $^{\textrm{rig}}\textbf{SPB}(\rho)_{P}^{n}$. By Theorem \ref{C} and Equation \ref{defiso}, there is a $\mathbb{C}^{*}$-invariant natural transformation
\begin{equation}\label{ahi}
^{\textrm{rig}}\textbf{SPB}(\rho)_{P}^{n}\rightarrow \  ^{\textrm{rig}}\textbf{Swamps}^{n}_{P,\mathscr{O}_{C}}.
\end{equation}
Moreover, the morphism induced between the representatives is a $\textrm{SL}(W)$-equivariant injective and proper morphism
\begin{equation}\label{injparam}
\beta\colon \mathbb{D}\slas \mathbb{C}^{*}\hookrightarrow Z'_{n,\Ox,P},
\end{equation}
$Z'_{n,\Ox,P}$ being the parameter space for swamps of the given type (see \cite{bosle, sols}, since the construction is the same).  Let $\mathbb{D}_{0}\subset \mathbb{D}$ be the open subscheme consisting of points such that $W\rightarrow H^{0}(\F(n))$ is an isomorphism and $\F$ is torsion free.
Then we have,
\begin{proposition}[Gluing Property]\label{hastalaminga2}
Let $S$ be a scheme of finite type over $\mathbb{C}$ and $s_{1},s_{2}:S\rightarrow\mathbb{D}_{0}$ two morphisms such that the pullbacks of $(\F_{\mathbb{D}_0},\tau_{\mathbb{D}_0})$ via $s_{1}\times \textrm{id}_{X}$ and $s_{2}\times \textrm{id}_{X}$ are isomorphic. Then there exists an \'etal\'e covering $c:T\rightarrow S$ and a morphism $h:T\rightarrow \emph{SL}(W)$ such the following triangle is commutative
$$
\xymatrix{
\emph{SL}(W)\times\mathbb{D}_{0}\ar[rr]^{\Gamma} & & \mathbb{D}_{0} \\
& T\ar[lu]^{h\times(s_{1}\circ c)}\ar[ru]_{s_{2}\circ c} &
}.
$$
\end{proposition}
\begin{proof}
Follows by the same argument given in \cite[Proposition 5.2]{Alexander2}.
\end{proof}

\begin{proposition}[Local Universal Property]\label{hastalaminga1}
Let $S$ be a scheme of finite type and $(\F_{S},\tau_{S})$ a family of $\delta$-(semi)stable  singular principal G-bundles parametrized by $S$. Then there exists an open covering $S_{i}, \ i\in I$ of $S$  and morphisms $\beta_{i}\colon S_{i}\rightarrow\mathbb{D}$, $i\in I$ such that the restriction of the family $(\F_{S},\tau_{S})$ to $S_{i}\times C$ is equivalent to the pullback of $(\F_{\mathbb{D}},\tau_{\mathbb{D}}) $ via $\beta_{i}\times \textrm{id}_{C}$ for all $i\in I$.
\end{proposition}
\begin{proof}
Follows by the same argument given in \cite[Proposition 5.1]{Alexander2}.
\end{proof}
Consider the linearized invertible sheaf $\mathcal{M}$ on $Z'_{n,\Ox,P}$ with respect to the moduli space of $\delta$-semistable swamps is constructed, and let $\LL:=\beta^{*}\mathcal{M}$. 
\begin{proposition}\label{pasitointermedio}  In the above situation, we have:

1) All semistable points with respect to $\LL$ lie in $\mathbb{D}_{0}$.

2) A point $y\in\mathbb{D}_{0}$ is (semi)stable with respecto to $\LL$ if and only if the restriction of the universal singular principal G-bundle to $\{y\}\times C$ is $\delta$-(semi)stable.
\end{proposition}
\begin{proof}
Follows from the construction of the morphism $\beta$ (see Equation \ref{injparam}), the definition of $\delta$-semistanility for singular principal bundles, and Theorem \ref{teolanger}.
\end{proof}

We finally have,
\begin{theorem}\label{B}
There is a projective scheme $\emph{SPB}(\rho)_{P}^{\delta\text{-(s)s}}$ and an open subscheme $\textrm{SPB}(\rho)_{P}^{\delta\text{-s}}\subset \textrm{\emph{SPB}}(\rho)_{P}^{\delta\text{-(s)s}}$ together with a natural transformation
\begin{equation*}
\alpha^{(s)s}\colon \textbf{\emph{SPB}}(\rho)_{P}^{\delta\textrm{-(s)s}}\rightarrow h_{\emph{SPB}(\rho)_{P}^{\delta\text{-(s)s}}}
\end{equation*}
with the following properties:

1) For every scheme $\N$ and every natural transformation $\alpha'\colon \emph{\textbf{SPB}}(\rho)_{P}^{\delta\textrm{-(s)s}}\rightarrow h_{\N}$, there exists a unique morphism $\varphi\colon \emph{SPB}(\rho)_{P}^{\delta\textrm{-(s)s}}\rightarrow\N$ with $\alpha'=h(\varphi)\circ\alpha^{(s)s}$.

2) The scheme $\emph{SPB}(\rho)_{P}^{\delta\textrm{-s}}$ is a coarse projective moduli space for the functor $\emph{\textbf{SPB}}_{P}(\rho)^{\delta\textrm{-s}}$
\end{theorem}
\begin{proof}
By Proposition \ref{hastalaminga2}, Proposition \ref{hastalaminga1}, Proposition \ref{pasitointermedio} and Theorem \ref{teolanger}, the quotients $\textrm{SPB}(\rho)_{P}^{\delta\textrm{-(s)s}}:=\mathbb{D}^{(s)s}\slas (\mathbb{C}^{*}\times\textrm{SL}(V))$ exist, $\mathbb{D}^{ss}\slas (\mathbb{C}^{*}\times \textrm{SL}(V))$ is a projective scheme and satisfies 1) and 2). 
\end{proof}

\subsection{The relative moduli space}

Let $S$ be a $\mathbb{C}$-scheme, $\pi:\CC\rightarrow S$ a flat family of nodal curves, $\textbf{Sch}_{S}$ the category of $S$-schemes, $\delta\in\mathbb{Q}_{>0}$ and $P(n)$ a polynomial of degree one with integral coefficients.  Consider now the moduli functor,

\begin{align*}
\textbf{SPB}_{\CC/S}(\rho)_{P}^{\delta\textrm{-(s)s}}:\textbf{Sch}_{S}&\longrightarrow \textbf{Set}\\
T&\longmapsto  \left \{   \begin{array}{c}\textrm{isomorphism classes of families of }\delta\textrm{-(semi)} \\  \textrm{stable singular principal }G\textrm{-bundles on }\CC\times_{S}T  \\  \textrm{ parametrized } \textrm{by } T\textrm{ with Hilbert polynomial }P \end{array}  \right\}
\end{align*}

Note that we can construct the relative parameter space for families of singular principal $G$-bundles copying almost word by word the construction of the parameter space for the fiberwise problem. With this in hand, we can establish the natural transformation (\ref{ahi}) corresponding to the relative problem due to uniform linearization given in Theorem \ref{C}. With this and Theorem \ref{teolanger} we can state the last theorem.

\begin{theorem}\label{D}
Let $\pi:\CC\rightarrow S$ be a flat family of nodal curves, $\Oo(1)$ a relative polarization with degree $h$ on the fibers, $G$ a (semisimple) reductive group, $V$ a vector space of dimension $r$, $\rho:G\hookrightarrow \emph{SL}(V)$ a fully faithful representation and $P$ a polynomial of degree one with integral coefficients such that leading coefficient is $h\cdot r$. Then there exists a coarse $S$-projective moduli space for $\delta$-semistable singular principal $G$-bundles, $(\F_S,\tau_S)$ such that $\F_S$ has Hilbert polynomial $P$ on the fibers.
\end{theorem}
\begin{remark}
Suppose that $S=\textrm{Spec} \ A$, $A$ being a discrete valuation ring. Assume also that the special fiber $\CC_{0}$ is a stable curve and the generic fiber $\CC_{\eta}$ is smooth.

1. Assume further that $\CC_0$ is irreducible. By the results of  \cite{Alexander1} we can assume that there is a rational number $\delta$ large enough such that both, $(\textrm{SPB}_{\CC/S}(\rho)_{P}^{\delta\textrm{-(s)s}})_\eta$ and $(\textrm{SPB}_{\CC/S}(\rho)_{P}^{\delta\textrm{-(s)s}})_0$, parametrizes semistable honest singular principal $G$-bundles. From \cite{Alexander2} we know that $(\textrm{SPB}_{\CC/S}(\rho)_{P}^{\delta\textrm{-(s)s}})_\eta$ coincides with the moduli space of semistable principal $G$-bundles. Theorem \ref{D} applied  to this situation was proved by G. Faltings in a more general way for the orthogonal and symplectic groups \cite{Fal}.

2. A natural question is whether the generic fiber $(\textrm{SPB}_{\CC/S}(\rho)_{P}^{\delta\textrm{-(s)s}})_\eta$ is dense or not. This was studied by X. Sun in a series of papers for the special linear group \cite{sun1,sun2} . A. Schmitt gave in \cite{Alexander1} a scheme-theoretic structure to the closure of the generic fiber in the whole space by using generalized parabolic structures, and it seems to be not clear, a priori, that this question has a positive answer.

\end{remark}

\section{\footnotesize A REMARK ON THE UNIVERSAL MODULI SPACE OVER THE MODULI SPACE OF STABLE CURVES}\label{section7}
\markboth{\emph{}}{} 

It is important to note that the consequences of Theorem \ref{C} are deeper than those deduced so far. To show that, let us recall briefly the construction of the moduli space of stable curves (see \cite{gies} for the whole construction) and apply Theorem \ref{D} to the universal curve of genus $g\geq 2$.

\subsection{Gieseker construction of $\overline{\mathcal{M}}_{g}$}

Fix integers $g\geq 2$, $d=10(2g-2)$ and $N=d-g$. Consider the Hilbert functor
\begin{equation*}
\textbf{Hilb}_{d,g}(S)=\left\{ 
\begin{array}{ccc}
Z\rightarrow & \hspace{-0.4cm}S \times\mathbf{P}^{N} & Z_{s}\textrm{ being a }\\
\hspace{0.5cm}\searrow & \hspace{-0.7cm}\downarrow  &  \textrm{, projective curve of genus } 
\\
& \hspace{-0.7cm}S & g \textrm{ and degree }d \ \forall s\in S
\end{array} 
 \right\}.
\end{equation*}

Let us denote by $H_{g,d,N}$ the representative of $\textbf{Hilb}_{d,g}$ and $H_{g}\subset H_{g,d,N}$ the locus of non-degenerate, 10-canonical Delign-Mumford stable curves of genus $g$. $H_{g}$ is a locally closed subscheme and it is a nonsingular, irreducible quasi-projective variety.

Recall that $N=10(2g-2)-g$, $\mathbb{C}^{N+1}\simeq H^{0}(C,\omega_{C}^{10})$ and $h(s)=ds-g+1$. Then $H_{g,d,N}$ is the Quot scheme,
$\textrm{Quot}^{h(s)}_{\mathscr{O}_{\mathbf{P}^{N}}/\mathbf{P}^{N}/ \mathbb{C}}$, 
and there exists an integer $s_{\alpha}$ such that $\forall s\geq s_{\alpha}$,
$i'_{s}:H_{d,g,N}\hookrightarrow \textrm{Grass}(h(s),H^{0}(\mathbf{P}^{N},\mathscr{O}_{\mathbf{P}^{N}}(s))^{\vee})$
is a closed immersion. Moreover, there exists $\overline{s}(g)$ such that the GIT linearized problem $i'_{\overline{s}(g)}$ satisfies:
(1) $H_{g}$ belongs to the semistable locus,
(2) $H_{g}$ is closed in the semistable locus.
Finally, the action of $\textrm{SL}_{N+1}$ on $\mathbf{P}^{N}$ induces an action on $H_{g}$, and Gieseker shows that
$
\overline{\mathcal{M}}_{g}=H_{g}/\textrm{SL}_{N+1}.
$

\subsection{A remark on the universal moduli space}

Let us denote by $\mathscr{U}_{g}$ the relative universal curve
\begin{equation}\label{uno}
\xymatrix{
\mathscr{U}_{g}\ar@{^(->}[rr]^{\small{\textrm{closed}}}_{\psi}\ar[rrd]^{\nu}_{\textrm{flat}}& & H_{g}\times\mathbf{P}^{N} \ar[r]^{d}\ar[d]^{e} & \mathbf{P}^{N} \\
& & H_{g} \ \ , & 
}
\end{equation}
and let $\nu:\mathscr{U}_{g}\rightarrow\mathbf{P}^{N}$ be the second projection. Consider the relative very ample line bundle $\nu^{*}\Oo_{\mathbf{P}^{N}}=:\Oo(1)$ on the universal curve. We can now apply Theorem \ref{D} to the family $\nu:\mathscr{U}_{g}\rightarrow H_{g}$ and the relative polarization $\Oo(1)$. Thus, we get a moduli space over $H_{g}$, $\textrm{SPB}_{\mathscr{U}_{g}/H_{g}}(\rho)_{P}^{\delta\textrm{-(s)s}}\rightarrow H_{g}$.
The $\textrm{SL}_{N+1}$-action on $H_{g}$ lifts to an $\textrm{SL}_{N+1}$-action on $\textrm{SPB}_{\mathscr{U}_{g}/H_{g}}(\rho)_{P}^{\delta\textrm{-(s)s}}$. Then we have to study the GIT problem associated to this $\textrm{SL}_{N+1}$-action if we want to construct the universal moduli space over $\overline{\mathcal{M}}_{g}$. Following \cite{Pando}, we need to show that all the numerical parameters involved in finding the right polarization on the parameter space for the fiberwise problem do not depend on the base curve. Therefore, the uniformity of the linearization given in Theorem \ref{C} reduces the construction of the universal moduli space for $\delta$-semistable singular principal $G$-bundles to the construction of the universal moduli space for $\delta$-semistable swamps. 

\begin{notation}
The results presented in this paper are part of the author's Ph.D. thesis at Freie Universit\"at Berlin. The author would like to thank Prof. Alexander Schmitt for his encouragement and support.
\end{notation}

\bibliographystyle{amsplain}
\bibliography{mibiblio2}

\end{document}